\documentclass[12pt,reqno]{amsart}
\usepackage{txfonts,fullpage}

\newtheorem{thm}{Theorem}[section]

\newtheorem{cor}[thm]{Corollary}
\newtheorem{rem}[thm]{Remark}

\numberwithin{equation}{section}
\allowdisplaybreaks

\title{A higher order Painlev\'e system in two variables and extensions of the Appell hypergeometric functions $F_1$, $F_2$ and $F_3$}
\date{}
\author{Takao Suzuki}
\address{Department of Mathematics, Kinki University, 3-4-1, Kowakae, Higashi-Osaka, Osaka 577-8502, Japan}
\email{suzuki@math.kindai.ac.jp}

\begin{document}

\maketitle

\begin{abstract}
In this article we propose an extension of Appell hypergeometric function $F_2$ (or equivalently $F_3$).
It is derived from a particular solution of a higher order Painlev\'e system in two variables.
On the other hand, an extension of Appell's $F_1$ was introduced by Tsuda.
We also show that those two extensions are equivalent at the level of systems of linear partial differential equations.


\end{abstract}

\section{Introduction}

Thanks to the previous works \cite{HF,Katz,Kos,O1,O2} we have a good classification theory of isomonodromy deformation equations of Fuchsian systems; we call them {\it higher order Painlev\'e systems}.
From this point of view several extensions of the Painlev\'e VI equation have been proposed in \cite{FISS,G,Kaw,Sak,Su2,T2}.

It is well known that the Painlev\'e VI equation (resp. the Garnier system) admits a particular solution in terms of the Gauss hypergeometric function ${}_2F_1$ (resp. Appell's $F_1$ or Lauricella's $F_D$); see \cite{Fuc,G,IKSY}.
In recent years an investigation of hypergeometric solutions of Painlev\'e systems has been developed greatly; see \cite{G,Su1,Su2,T1}.
We list the obtained hypergeometric solutions of 4th and 6th order Painlev\'e systems in the following table.
\[\begin{array}{|l|l|l|}\hline
	\text{Painlev\'e system} & \text{Rigid system} & \text{HGF} \\\hline
	\mathcal{H}^{11\times5} \ (=\mathcal{H}^{21\times5}) & 21\times4/21\times4 \ (P_3/P_3) & F_1 \\
	\mathcal{H}^{21,21,111,111} & 21,111,111 & {}_3F_2 \\\hline
	\mathcal{H}^{11\times6} \ (=\mathcal{H}^{31\times6}) & 31\times5/31\times5/31\times5 \ (P_4/P_4/P_4) & F_D \\
	\mathcal{H}^{21,21,21,21,111} \ (=\mathcal{H}^{31,31,31,22,211}) & 31,31,22,211/31,31,22,211 \ (I_4/I_4,\mathrm{II}^*_2/\mathrm{II}^*_2) & F_2,F_3 \\
	\mathcal{H}^{31,31,22,22,22} & 31,22,22,22/31,22,22,22 \ (P_{4,4}/P_{4,4}) & F_4 \\
	\mathcal{H}^{21,111,111,111} \ (=\mathcal{H}^{31,211,211,211}) & 211,211,211 \ (\mathrm{II}_2) &  \\
	\mathcal{H}^{31,22,211,1111} & 22,211,1111 \ (EO_4) &  \\
	\mathcal{H}^{31,31,1111,1111} & 31,1111,1111 & {}_4F_3 \\\hline
\end{array}\]
The symbol $\mathcal{H}^{\lambda}$ stands for a Painlev\'e system which is derived from the Fuchsian system with the spectral type $\lambda$.
Note that the Garnier system in two (resp. three) variables is denoted by $\mathcal{H}^{11\times5}$ (resp. $\mathcal{H}^{11\times6}$).
The equations $P_3,P_4$ are members of the Jordan-Pochhammer's family, $I_4,\mathrm{II}_2,\mathrm{II}^*_2$ are in \cite{Y}, $P_{4,4}$ is in \cite{O2} and $EO_4$ is a member of the even family appeared in \cite{Sim}.

\begin{rem}
A spectral type is a multiplicity data of the characteristic exponents (or the eigenvalues of the residue matrices).
In this article we introduce one example instead of the definition.
A spectral type of the Fuchsian system with the Riemann scheme
\[
	\left\{\begin{array}{cccc}
		x=t & x=1 & x=0 & x=\infty \\
		\theta_1 & \theta_2 & \kappa_1 & \rho_1 \\
		0 & 0 & \kappa_1 & \rho_1 \\
		0 & 0 & 0 & \rho_2 \\
		0 & 0 & 0 & \rho_3 \\
	\end{array}\right\},
\]
is $\{(3,1),(3,1),(2,2),(2,1,1)\}$.
In the above table we denote it by $31,31,22,211$ for sake of simplicity.
We also denote $21,21,21,21$ by $21\times4$.
\end{rem}

\begin{rem}
Since the Fuchsian system of type $11\times5$ is transformed by Katz's two operations, namely addition and middle convolution, to the one of type $21\times5$, the Painlev\'e system $\mathcal{H}^{11\times5}$ is equivalent to $\mathcal{H}^{21\times5}$; see \cite{HF}.
It is the same for the other systems.
\end{rem}

\begin{rem}
In the case of Painlev\'e systems in multi-variables the notation of rigid systems is interpreted as follows.
We choose $\mathcal{H}^{11\times5}$, namely the Garnier system in two variables, as an example.
Under a certain specialization the system $\mathcal{H}^{11\times5}$ is reduced to a linear Pfaff one in two variables whose solution is described in terms of Appell's $F_1$.
If we regard one variable as an independent variable and another as a constant, then this Pfaff system becomes a rigid one of type $21\times4$.
\end{rem}

Let $\lambda_i$ $(i=2,\ldots,m+3)$ be partitions of a natural number $n+1$.
By an observation of the above table we expect that the $2n$-th order Painlev\'e system $\mathcal{H}^{n1,\lambda_2,\ldots,\lambda_{m+3}}$ admits a particular solution in terms of the $(n+1)$-st order rigid system of type $\lambda_2,\ldots,\lambda_{m+3}$.
It is probably true in the case of $m=1$, namely the Painlev\'e system is an ordinary differential one; see Section 5 of \cite{Su2}.
However, in the case of $m\geq2$, the situation is more complicated.
As is seen later, a Painlev\'e system (or a linear Pfaff system) in multi-variables doesn't always contain only one type of rigid system.
Therefore we have to investigate the case of $m\geq2$ more deeply in order to go toward a unified theory of Painlev\'e systems, rigid systems and hypergeometric functions.

As a first step of our purpose we investigate {\it the Schlesinger-Tsuda system} $\mathcal{H}_{n+1,2}$ given in \cite{T2} and its particular solution.
The system $\mathcal{H}_{n+1,2}$ is derived from the Fuchsian system of type $\{(n,1),(n,1),(n,1),(1,\ldots,1),(1,\ldots,1)\}$.
According to Oshima's work \cite{O1,O2}, this Fuchsian system is transformed by Katz's two operations to the following two types of systems.
\[\begin{array}{c}
	\{(2n,1),(2n,1),(n+1,n),(n+1,1,\ldots,1),(n+1,1,\ldots,1)\} \\
	\{(2n,1),(2n,1),(2n,1),(2,\ldots,2,1),(2,\ldots,2,1)\}
\end{array}\]
Hence we expect that the following two types of rigid systems appear in a particular solution of $\mathcal{H}_{n+1,2}$.
\[\begin{array}{rl}
	I_{2n+1}: & \{(2n,1),(n+1,n),(n+1,1,\ldots,1),(n+1,1,\ldots,1)\} \\
	J_{2n+1}: & \{(2n,1),(2n,1),(2,\ldots,2,1),(2,\ldots,2,1)\}
\end{array}\]
In \cite{T1,T3} the Pfaff system $I_{2n+1}/I_{2n+1}$ has already derived from $\mathcal{H}_{n+1,2}$.
In this article we look for another particular solution of $\mathcal{H}_{n+1,2}$ and derive a Pfaff system $J_{2n+1}/I_{2n+1}$.
We also present a hypergeometric function
\[
	F_2^{(n)}\left[\begin{array}{c}b_1,\ldots,b_n,b',a\\c_1,\ldots,c_n,c'\end{array};t_1,t_2\right] = \sum_{i,j=0}^{\infty}\frac{(b_1)_i\ldots(b_n)_i(b')_j(a)_{i+j}}{(c_1)_i\ldots(c_n)_i(c')_j(1)_i(1)_j}t_1^it_2^j,
\]
as an extension of Appell's $F_2$ (or equivalently $F_3$).
The solution of $J_{2n+1}/I_{2n+1}$ is described in terms of $F_2^{(n)}$ with $a=c_1+c'-2$.
Note that the function $F_2^{(n)}$ without constraints between parameters gives the solution of a Pfaff system $J_{2n+2}/I_{2n+2}$.

\begin{rem}
In this article we identify the hypergeometric functions $F_3$ with $F_2$.
In fact, the system of linear partial differential equations (LPDEs) which $F_2$ satisfies
\begin{align*}
	&\{t_1(\delta_1+\delta_2+a)(\delta_1+b)-\delta_1(\delta_1+c-1)\}z = 0,\\
	&\{t_2(\delta_1+\delta_2+a)(\delta_2+b')-\delta_2(\delta_2+c'-1)\}z = 0,
\end{align*}
where $\delta_1=t_1\partial/\partial t_1$, $\delta_2=t_2\partial/\partial t_2$, is transformed by $(t_1,t_2)\to(1/t_1,1/t_2)$ to the system which $F_3$ satisfies.
Then each solution of transformed $F_3$-system is a linear combination of four independent solutions of $F_2$-system; see \cite{E}.
\end{rem}

On the other hand, the hypergeometric function $F_{n+1,2}$ was introduced by Tsuda as an extension of Appell's $F_1$ in \cite{T1}.
The solution of $I_{2n+1}/I_{2n+1}$ is described in terms of $F_{n+1,2}$; see \cite{T3}.
As the system $I_{2n+1}/I_{2n+1}$ is transformed by $(t_1,t_2)\to(1/t_1,t_2/t_1)$ to $J_{2n+1}/I_{2n+1}$, we can predict that the function $F_2^{(n)}$ with $a=c_1+c'-2$ is equivalent to $F_{n+1,2}$.
In this article it is shown to be true at the level of systems of LPDEs.
The relationship between $F_2$ and $F_1$ was pointed out in \cite{B,Sla}.
Our result becomes an extension of this classical result.

This article is organized as follows.
In Section \ref{Sec:Gen_Appell} a hypergeometric function $F_2^{(n)}$ is defined by a formal power series.
We also give a system of LPDEs and an integral expression for $F_2^{(n)}$.
In Section \ref{Sec:Sch_HGE} a Pfaff system $J_{2n+1}/I_{2n+1}$ is derived from $\mathcal{H}_{n+1,2}$ as a particular solution.
It is shown in Section \ref{Sec:GA_deg_Pfaff_Proof} that this Pfaff system admits a solution in terms of $F_2^{(n)}$.
In Section \ref{Sec:Sch_HGE_deg} a degeneration structure of the particular solution of $\mathcal{H}_{n+1,2}$ is investigated.
In Section \ref{Sec:Rel_F2_F1} a relationship between $F_2^{(n)}$ and $F_{n+1,2}$ is investigated from two points of view, namely a system of LPDEs and an integral expression.
In Section \ref{Sec:Conclusion} we state a summary of this article and some future problem.

\section{An extension of the Appell hypergeometric function $F_2$}\label{Sec:Gen_Appell}

We define a hypergeometric function $F_2^{(n)}$ as an extension of Appell's $F_2$ by a formal power series
\[
	F_2^{(n)}\left[\begin{array}{c}b_1,\ldots,b_n,b',a\\c_1,\ldots,c_n,c'\end{array};t_1,1-t_2\right] = \sum_{i,j=0}^{\infty}\frac{(b_1)_i\ldots(b_n)_i(b')_j(a)_{i+j}}{(c_1)_i\ldots(c_n)_i(c')_j(1)_i(1)_j}t_1^i(1-t_2)^j,
\]
which is convergent in $|t_1|+|1-t_2|<1$.
This series satisfies a system of linear partial differential equations
\begin{equation}\begin{split}\label{Eq:Gen_Appell_2n+2}
	&[t_1(\delta_1+\delta_2+a)\{\prod_{i=1}^{n}(\delta_1+b_i)\}-\delta_1\{\prod_{i=1}^{n}(\delta_1+c_i-1)\}]z = 0,\\
	&\{(1-t_2)(\delta_1+\delta_2+a)(\delta_2+b')-\delta_2(\delta_2+c'-1)\}z = 0,
\end{split}\end{equation}
where $\delta_1=t_1\partial/\partial t_1$, $\delta_2=(t_2-1)\partial/\partial t_2$, 
An integral expression of $F_2^{(n)}$ is given by
\begin{equation}\begin{split}\label{Eq:Gen_Appell_int}
	&F_2^{(n)}\left[\begin{array}{c}b_1,\ldots,b_n,b',a\\c_1,\ldots,c_n,c'\end{array};t_1,1-t_2\right]\\
	&= t_1^{1-c_1}(1-t_2)^{1-c'}\{\prod_{k=1}^{n}\frac{\Gamma(c_k)}{\Gamma(b_k)\Gamma(c_k-b_k)}\}\frac{\Gamma(c')}{\Gamma(b')\Gamma(c'-b')}\int_{0<u_n<\ldots<u_1<t_1,1<u_{n+1}<t_2}U(t)du_1\ldots du_{n+1},
\end{split}\end{equation}
where
\begin{align*}
	U(t) &= (t_1-u_1)^{c_1-b_1-1}\{\prod_{k=1}^{n-1}u_k^{b_k-c_{k+1}}(u_k-u_{k+1})^{c_{k+1}-b_{k+1}-1}\}\\
	&\quad\times u_n^{b_n-1}(u_n-u_{n+1})^{-a}(t_2-u_{n+1})^{c'-b'-1}(1-u_{n+1})^{b'-1}.
\end{align*}
The symbol
\[
	\int_{0<u_n<\ldots<u_1<t_1,1<u_{n+1}<t_2}du_1\ldots du_{n+1},
\]
stands for an iterated integral
\[
	\int_{0}^{t_1}\int_{0}^{u_1}\ldots\int_{0}^{u_{n-1}}\int_{1}^{t_2}du_{n+1}du_n\ldots du_2du_1.
\]
We will prove \eqref{Eq:Gen_Appell_int} in Appendix \ref{App:Gen_Appell_int}.

\begin{rem}
The function $F_2^{(n)}$ is a special case of the Kamp\'e de F\'eriet function given in \cite{AK}.
However the Kamp\'e de F\'eriet function is a wide class of functions and hence it is not obvious to choose a suitable function among them.
\end{rem}

\begin{rem}
A solution of the rigid systems $J_{2n+2}$ and $I_{2n+1}$ has been already given in the form of a formal power series and an integral expression in \cite{O2}.
However our function $F_2^{(n)}$ is more suitable to clarify a relationship with $F_2$.
\end{rem}

If, in system \eqref{Eq:Gen_Appell_2n+2}, we assume that $a=c_1+c'-2$, then we obtain
\begin{equation}\begin{split}\label{Eq:Gen_Appell_2n+1}
	&[t_1(\delta_1+\delta_2+c_1+c'-2)\{\prod_{i=1}^{n}(\delta_1+b_i)\}-\delta_1\{\prod_{i=1}^{n}(\delta_1+c_i-1)\}]z = 0,\\
	&\{(1-t_2)(\delta_1+\delta_2+c_1+c'-2)(\delta_2+b')-\delta_2(\delta_2+c'-1)\}z = 0.
\end{split}\end{equation}
System \eqref{Eq:Gen_Appell_2n+1} can be rewritten to a Pfaff one $J_{2n+1}/I_{2n+1}$; see Section \ref{Sec:Sch_HGE}.
We now choose a solution of \eqref{Eq:Gen_Appell_2n+1}
\[
	z = (1-t_2)^{1-c'}F_2^{(n)}\left[\begin{array}{c}b_1,\ldots,b_n,b'-c'+1,c_1-1\\c_1,\ldots,c_n,2-c'\end{array};t_1,1-t_2\right].
\]
Then we obtain an integral expression of $(\delta_1+c_1-1)z$ as
\begin{equation}\begin{split}\label{Eq:Gen_Appell_deg_int}
	&(1-t_2)^{1-c'}F_2^{(n)}\left[\begin{array}{c}b_1,\ldots,b_n,b'-c'+1,c_1-1\\c_1-1,c_2,\ldots,c_n,2-c'\end{array};t_1,1-t_2\right]\\
	&= t_1^{1-c_2}\{\prod_{k=2}^{n}\frac{\Gamma(c_k)}{\Gamma(b_k)\Gamma(c_k-b_k)}\}\frac{\Gamma(2-c')}{\Gamma(b'-c'+1)\Gamma(1-b')}\int_{0<u_{n-1}<\ldots<u_1<t_1,1<u_n<t_2}U(t)du_1\ldots du_n,
\end{split}\end{equation}
where
\begin{align*}
	U(t) &= (t_1-u_1)^{c_2-b_2-1}\{\prod_{k=2}^{n-1}u_{k-1}^{b_k-c_{k+1}}(u_{k-1}-u_k)^{c_{k+1}-b_{k+1}-1}\}\\
	&\quad\times u_{n-1}^{b_n-1}(u_{n-1}-u_n)^{-b_1}(t_2-u_n)^{-b'}(1-u_n)^{b'-c'}u_n^{b_1-c_1+1}.
\end{align*}
We will prove \eqref{Eq:Gen_Appell_deg_int} in Appendix \ref{App:Gen_Appell_int}.

In the last we state a degeneration from $F_2^{(n)}$ to $F_2^{(n-1)}$ briefly.
Let $\tilde{z}=(\delta_1+c_1-1)z$ and assume that $b_1=c_1-1$.
Then system \eqref{Eq:Gen_Appell_2n+1} can be reduced to
\begin{equation}\begin{split}\label{Eq:Gen_Appell_2n}
	&[t_1(\delta_1+\delta_2+c_1+c'-2)\{\prod_{i=2}^{n}(\delta_1+b_i)\}-\delta_1\{\prod_{i=2}^{n}(\delta_1+c_i-1)\}]\tilde{z} = 0,\\
	&\{(1-t_2)(\delta_1+\delta_2+c_1+c'-2)(\delta_2+b')-\delta_2(\delta_2+c'-1)\}\tilde{z} = 0.
\end{split}\end{equation}
As is seen in Section \ref{Sec:Sch_HGE_deg}, system \eqref{Eq:Gen_Appell_2n} can be rewritten to a Pfaff one $J_{2n}/I_{2n}$.

\section{The Schlesinger-Tsuda system $\mathcal{H}_{n+1,2}$ and its particular solution in terms of $J_{2n+1}/I_{2n+1}$}\label{Sec:Sch_HGE}

The Schlesinger-Tsuda system $\mathcal{H}_{n+1,2}$ given in \cite{T2} is expressed as the Hamiltonian system of $4n$-th order in two variables
\[
	\frac{\partial q_j}{\partial t_i} = \frac{\partial H_i}{\partial p_j},\quad
	\frac{\partial p_j}{\partial t_i} = -\frac{\partial H_i}{\partial q_j},\quad
	\frac{\partial q'_j}{\partial t_i} = \frac{\partial H_i}{\partial p'_j},\quad
	\frac{\partial p'_j}{\partial t_i} = -\frac{\partial H_i}{\partial q'_j}\quad (i=1,2;j=1,\ldots,n).
\]
The Hamiltonian $H_1$ is given by
\begin{align*}
	H_1 &= \frac{H_1^{(t)}}{t_1-t_2}+\frac{H_1^{(1)}}{t_1-1}+\frac{H_1^{(0)}}{t_1},\\
	H_1^{(t)} &= \{\sum_{i=1}^{n}q_i(p_i-p'_i)+\theta_1\}\{\sum_{i=1}^{n}q'_i(p'_i-p_i)+\theta_2\},\\
	H_1^{(1)} &= [\sum_{i=1}^{n}\{(q_i-1)p_i+q'_ip'_i-\kappa_i-\rho_i\}-\theta_3][\sum_{i=1}^{n}q_i\{(q_i-1)p_i+q'_ip'_i-\kappa_i-\rho_i\}-\theta_1],\\
	H_1^{(0)} &= (\sum_{i=1}^{n}q_ip_i+\theta_1)\{\sum_{i=1}^{n}(q_i+q'_i-1)p_i+\kappa_0\} - \sum_{i=1}^{n}\kappa_iq_ip_i\\
	&\quad + \sum_{i=1}^{n}\sum_{j=i+1}^{n}q_ip_j\{(q_i-q_j)p_i+(q'_i-q'_j)p'_i-\kappa_i-\rho_i\}.
\end{align*}
The Hamiltonian $H_2$ is obtained from $H_1$ via a replacement
\[
	t_1\leftrightarrow t_2,\quad
	p_i\leftrightarrow p'_i,\quad
	q_i\leftrightarrow q'_i,\quad
	\theta_1\leftrightarrow\theta_2.
\]

\begin{rem}
Although the system $\mathcal{H}_{n+1,2}$ contains $2n+3$ parameters, we have $2n+4$ parameters, namely $\theta_1,\theta_2,\theta_3,\kappa_0,\ldots,\kappa_n,\rho_1,\ldots,\rho_n$, in the above Hamiltonian.
In fact, we can let any one of $\kappa_0,\ldots,\kappa_n$ be equal to zero without loss of generality.
In this article we leave all of those parameters in order to write some formulas concisely.
\end{rem}

In the system $\mathcal{H}_{n+1,2}$ we assume that
\begin{equation}\label{Eq:Sch_Tsu_PS_F2}
	q_ip_i-\kappa_i-\rho_i = 0,\quad p'_i = 0\quad (i=1,\ldots,n),\quad \theta_1+\sum_{j=1}^{n}(\kappa_j+\rho_j) = 0.
\end{equation}
We also define a dependent variable $w_0$ by
\begin{align*}
	d\log w_0 &= (\sum_{j=1}^{n}p_j+\theta_3)d\log(t_1-1) + \{\sum_{j=1}^{n}(q'_j-1)p_j+\kappa_0+\rho_1\}d\log t_1\\
	&\quad + (-\sum_{j=1}^{n}q'_jp_j+\theta_2)d\log(t_1-t_2),
\end{align*}
where $d$ stands for an exterior derivative for $t_1,t_2$, and set
\[
	w_i = p_iw_0,\quad
	w'_i = -q'_ip_iw_0\quad (i=1,\ldots,n).
\]
Then we obtain the following theorem by a direct computation.

\begin{thm}
A vector of dependent variables $\mathbf{w}={}^t[w_0,w_1,\ldots w_n,w'_1,\ldots,w'_n]$ defined above satisfies a linear Pfaff system
\begin{equation}\label{Eq:Sch_Tsu_Pfaff}
	d\mathbf{w} = \{A_1^{(1)}d\log(t_1-1)+A_0^{(1)}d\log t_1+A_td\log(t_1-t_2)+A_1^{(2)}d\log(t_2-1)+A_0^{(2)}d\log t_2\}\mathbf{w},
\end{equation}
with matrices
\begin{align*}
	A_1^{(1)} &= \theta_3E_{0,0}+\sum_{j=1}^{n}E_{0,j} + \sum_{i=1}^{n}\{\theta_3(\kappa_i+\rho_i)E_{i,0}+\sum_{j=1}^{n}(\kappa_i+\rho_i)E_{i,j}\},\\
	A_0^{(1)} &= (\kappa_0+\rho_1)E_{0,0}-\sum_{j=1}^{n}E_{0,j}-\sum_{j=1}^{n}E_{0,n+j} + \sum_{i=1}^{n}\{(\rho_1-\rho_i)E_{i,i}-\sum_{j=i+1}^{n}(\kappa_i+\rho_i)E_{i,j}\}\\
	&\quad + \sum_{i=1}^{n}\{(\rho_1-\rho_i)E_{n+i,n+i}-\sum_{j=i+1}^{n}(\kappa_i+\rho_i)E_{n+i,n+j}\},\\
	A_t &= \theta_2E_{0,0}+\sum_{j=1}^{n}E_{0,n+j} + \sum_{i=1}^{n}\{\theta_2(\kappa_i+\rho_i)E_{n+i,0}+\sum_{j=1}^{n}(\kappa_i+\rho_i)E_{n+i,n+j}\},\\
	A_1^{(2)} &= \sum_{i=1}^{n}\{\theta_2E_{i,i}-\theta_3E_{i,n+i}\} + \sum_{i=1}^{n}\{-\theta_2E_{n+i,i}+\theta_3E_{n+i,n+i}\},\\
	A_0^{(2)} &= \sum_{i=1}^{n}\{-\theta_2(\kappa_i+\rho_i)E_{n+i,0}+\theta_2E_{n+i,i} -\sum_{j=1}^{i-1}(\kappa_i+\rho_i)E_{n+i,n+j}+(\theta_2+\kappa_0-\kappa_i)E_{n+i,n+i}\},
\end{align*}
where $E_{i,j}$ stands for the $(2n+1)\times(2n+1)$ matrix with $1$ in the $(i,j)$-th entry and zeros elsewhere.
\end{thm}

The Riemann scheme of system \eqref{Eq:Sch_Tsu_Pfaff} is given by
\[
	\left\{\begin{array}{ccccccc}
		t_1=1 & t_1=0 & t_1=\infty & t_1=t_2 & t_2=1 & t_2=0 & t_2=\infty \\
		{[0]_{2n}} & {[0]_2} & {[-\kappa_1-\rho_1]_2} & {[0]_{2n}} & {[0]_{n+1}} & {[0]_{n+1}} & {[-\theta_2]_{n+1}} \\
		\theta_3' & {[\rho_1-\rho_2]_2} & \vdots & \theta_2' & {[\theta_2+\theta_3]_{n}} & \kappa_1' & \rho_1' \\
		& \vdots & {[-\kappa_n-\rho_1]_2} & & & \vdots & \vdots \\
		& {[\rho_1-\rho_n]_2} & \rho_1' & & & \kappa_n' & \rho_n' \\
		& \kappa_0+\rho_1 & & & & & \\
	\end{array}\right\},
\]
where
\begin{align*}
	&\theta_2' = \theta_2 + \sum_{j=1}^{n}(\kappa_j+\rho_j),\quad
	\theta_3' = \theta_3 + \sum_{j=1}^{n}(\kappa_j+\rho_j),\\
	&\kappa_i' = \theta_2+\kappa_0-\kappa_i,\quad
	\rho_i' = -\theta_2-\theta_3-\kappa_0-\rho_i\quad (i=1,\ldots,n),
\end{align*}
and the symbol ${[a]_{k}}$ means that the multiplicity of the eigenvalue $a$ is $k$.
Hence we can find that system \eqref{Eq:Sch_Tsu_Pfaff} is the Pfaff one $J_{2n+1}/I_{2n+1}$.

A solution of system \eqref{Eq:Sch_Tsu_Pfaff} can be described in terms of the degeneration of the function $F_2^{(n)}$ given in Section \ref{Sec:Gen_Appell}.

\begin{thm}\label{Thm:GA_deg_Pfaff}
Let $z$ be a solution of system \eqref{Eq:Gen_Appell_2n+1}.
We also set
\begin{equation}\begin{split}\label{Eq:GA_deg_Pfaff}
	w_0 &= -\{\prod_{j=1}^{n}(\delta_1+b_j)\}(\delta_2+c'-1)z,\\
	w_1 &= -b_1(b'-c'+1)\{\prod_{j=2}^{n}(\delta_1+b_j)\}(\delta_1+c_1-1)z,\\
	w_i &= -(b_i-c_i+1)(b'-c'+1)\{\prod_{j=i+1}^{n}(\delta_1+b_j)\}\{\prod_{j=1}^{i-1}(\delta_1+c_j-1)\}\delta_1z\quad (i=2,\ldots,n),\\
	w'_1 &= b_1\{\prod_{j=2}^{n}(\delta_1+b_j)\}(\delta_1+c_1-1)(\delta_2+b')z,\\
	w'_i &= (b_i-c_i+1)\{\prod_{j=i+1}^{n}(\delta_1+b_j)\}\{\prod_{j=1}^{i-1}(\delta_1+c_j-1)\}\delta_1(\delta_2+b')z\quad (i=2,\ldots,n).
\end{split}\end{equation}
Then a vector of dependent variables $\mathbf{w}={}^t[w_0,w_1,\ldots w_n,w'_1,\ldots,w'_n]$ satisfies system \eqref{Eq:Sch_Tsu_Pfaff} under a change of parameters
\begin{align*}
	&b_i = -\kappa_i-\rho_1\quad (i=1,\ldots,n),\quad
	b' = -\theta_2,\\
	&c_1 = -\kappa_0-\rho_1+1,\quad
	c_i = -\rho_1+\rho_i+1\quad (i=2,\ldots,n),\quad
	c' = 1-\theta_2-\theta_3.
\end{align*}
\end{thm}

We will prove this theorem in Section \ref{Sec:GA_deg_Pfaff_Proof}.

\begin{rem}\label{Rem:Rel_F2_F1}
If, in the system $\mathcal{H}_{n+1,2}$, we assume that
\begin{equation}\label{Eq:Sch_Tsu_PS_F1}
	p_i=p'_i=0\quad (i=1,\ldots,n),\quad
	\theta_3+\sum_{j=1}^{n}(\kappa_j+\rho_j)=0,
\end{equation}
then we obtain the Pfaff system $I_{2n+1}/I_{2n+1}$ given in \cite{T1,T3}.
As a matter of fact, a relationship between two specializations \eqref{Eq:Sch_Tsu_PS_F2} and \eqref{Eq:Sch_Tsu_PS_F1}, or equivalently two Pfaff systems $J_{2n+1}/I_{2n+1}$ and $I_{2n+1}/I_{2n+1}$, is derived from the birational transformation
\begin{align*}
	&q_i \to \frac{1}{q_i},\quad
	p_i \to -q_i(q_ip_i+q'_ip'_i-\kappa_i-\rho_i),\quad
	q'_i \to -\frac{q'_i}{q_i},\quad
	p'_i \to -q_ip'_i,\\
	&\theta_1 \to \theta_3,\quad
	\theta_3 \to \theta_1,\quad
	t_1 \to \frac{1}{t_1},\quad
	t_2 \to \frac{t_2}{t_1}\quad (i=1,\ldots,n),
\end{align*}
under which the system $\mathcal{H}_{n+1,2}$ is invariant; see \cite{T2}.
\end{rem}

\section{Proof of Theorem \ref{Thm:GA_deg_Pfaff}}\label{Sec:GA_deg_Pfaff_Proof}

In this section we show that the vector of dependent variables $\mathbf{w}={}^t[w_0,w_1,\ldots w_n,w'_1,\ldots,w'_n]$ defined by \eqref{Eq:GA_deg_Pfaff} satisfies partial differential equations
\begin{align}
	\delta_1w_0 &= \frac{t_1}{t_1-t_2}(-b'w_0+\sum_{j=1}^{n}w'_j) + \frac{t_1}{t_1-1}\{(b'-c'+1)w_0+\sum_{j=1}^{n}w_j\} - \widetilde{c_0}w_0 - \sum_{j=1}^{n}(w_j+w'_j),\label{Eq:GA_deg_Pfaff_Proof_w01}\\
	\delta_1w_i &= -\frac{t_1}{t_1-1}\widetilde{b_i}\{(b'-c'+1)w_0+\sum_{j=1}^{n}w_j\} - \widetilde{c_i}w_i + \sum_{j=i+1}^{n}\widetilde{b_i}w_j,\label{Eq:GA_deg_Pfaff_Proof_wi1}\\
	\delta_1w'_i &= -\frac{t_1}{t_1-t_2}\widetilde{b_i}(-b'w_0+\sum_{j=1}^{n}w'_j) - \widetilde{c_i}w'_i + \sum_{j=i+1}^{n}\widetilde{b_i}w'_j,\label{Eq:GA_deg_Pfaff_Proof_w'i1}\\
	\delta_2w_0 &= -\frac{t_2-1}{t_2-t_1}(b'w_0-\sum_{j=1}^{n}w'_j),\label{Eq:GA_deg_Pfaff_Proof_w02}\\
	\delta_2w_i &= -b'w_i - (b'-c'+1)w'_i,\label{Eq:GA_deg_Pfaff_Proof_wi2}\\
	\delta_2w'_i &= \frac{t_2-1}{t_2-t_1}\widetilde{b_i}(b'w_0-\sum_{j=1}^{n}w'_j) + \frac{t_2-1}{t_2}\{-\widetilde{b_i}b'w_0-b'w_i+\sum_{j=1}^{i-1}\widetilde{b_i}w'_j+(b_i-c_1-b'+1)w'_i\}\nonumber\\
	&\quad + b'w_i + (b'-c'+1)w'_i,\label{Eq:GA_deg_Pfaff_Proof_w'i2}
\end{align}
for $i=1,\ldots,n$, where
\[
	\widetilde{c_0} = c_1-1,\quad
	\widetilde{b_1} = b_1,\quad
	\widetilde{c_1} = 0,\quad
	\widetilde{b_j} = b_j-c_j+1,\quad
	\widetilde{c_j} = c_j-1\quad (j=2,\ldots,n).
\]
Equation \eqref{Eq:GA_deg_Pfaff_Proof_wi2} follows from definition \eqref{Eq:GA_deg_Pfaff} immediately.
We will prove the other five equations.

We first show that
\begin{equation}\label{Eq:GA_deg_Pfaff_Proof_w11}
	\delta_1w_1 = -\frac{t_1}{t_1-1}b_1\{(b'-c'+1)w_0+\sum_{j=1}^{n}w_j\} + \sum_{j=2}^{n}b_1w_j.
\end{equation}
Substituting system \eqref{Eq:Gen_Appell_2n+1} to definition \eqref{Eq:GA_deg_Pfaff}, we have
\begin{equation}\begin{split}\label{Eq:GA_deg_Pfaff_Proof_w11_Proof_1}
	(\delta_1+b_1)w_1 &= -b_1(b'-c'+1)\{\prod_{j=1}^{n}(\delta_1+b_j)\}(\delta_1+\widetilde{c_0})z\\
	&= b_1(b'-c'+1)\{\prod_{j=1}^{n}(\delta_1+b_j)\}(\delta_2+c'-1) - \frac{1}{t_1}b_1(b'-c'+1)\{\prod_{j=0}^{n}(\delta_1+\widetilde{c_j})\}z\\
	&= -b_1(b'-c'+1)w_0 - \frac{1}{t_1}b_1(b'-c'+1)\{\prod_{j=0}^{n}(\delta_1+\widetilde{c_j})\}z.
\end{split}\end{equation}
On the other hand, we have
\begin{align*}
	&(b'-c'+1)(\delta_1+\widetilde{c_0})\{\prod_{j=1}^{n}(\delta_1+\widetilde{c_j})-\prod_{j=1}^{n}(\delta_1+b_j)\}z\\
	&= w_1 + (b'-c'+1)(\delta_1+\widetilde{c_0})\delta_1\{\prod_{j=2}^{n}(\delta_1+\widetilde{c_j})-\prod_{j=2}^{n}(\delta_1+b_j)\}z\\
	&= w_1 + w_2 + (b'-c'+1)(\delta_1+\widetilde{c_0})\delta_1(\delta_1+\widetilde{c_2})\{\prod_{j=3}^{n}(\delta_1+\widetilde{c_j})-\prod_{j=3}^{n}(\delta_1+b_j)\}z\\
	&= w_1 + w_2 + \ldots + w_{n-1} + (b'-c'+1)\prod_{j=0}^{n-1}(\delta_1+\widetilde{c_j})\{(\delta_1+\widetilde{c_n})-(\delta_1+b_n)\}z\\
	&= \sum_{j=1}^{n}w_j.
\end{align*}
It implies
\begin{equation}\label{Eq:GA_deg_Pfaff_Proof_w11_Proof_2}
	(b'-c'+1)w_0 + \sum_{j=1}^{n}w_j = \frac{t_1-1}{t_1}(b'-c'+1)\prod_{j=0}^{n}(\delta_1+\widetilde{c_j})z.
\end{equation}
Note that we use
\[
	w_0 = \{(\delta_1+\widetilde{c_0})\prod_{j=1}^{n}(\delta_1+b_j)-\frac{1}{t_1}(\delta_1+\widetilde{c_0})\prod_{j=1}^{n}(\delta_1+\widetilde{c_j})\}z,
\]
in the above calculation.
Combining equations \eqref{Eq:GA_deg_Pfaff_Proof_w11_Proof_1} and \eqref{Eq:GA_deg_Pfaff_Proof_w11_Proof_2}, we obtain \eqref{Eq:GA_deg_Pfaff_Proof_w11}.

We next show that
\begin{equation}\label{Eq:GA_deg_Pfaff_Proof_w'11}
	\delta_1w'_1 = -\frac{t_1}{t_1-t_2}b_1(-b'w_0+\sum_{j=1}^{n}w'_j) + \sum_{j=2}^{n}b_1w'_j.
\end{equation}
An equation
\[
	-(b'-c'+1)w'_j = (\delta_2+b')w_j,
\]
follows from definition \eqref{Eq:GA_deg_Pfaff}.
It implies
\begin{align*}
	\sum_{j=1}^{n}w'_j &= -(\delta_2+b')(\delta_1+\widetilde{c_0})\{\prod_{j=1}^{n}(\delta_1+\widetilde{c_j})-\prod_{j=1}^{n}(\delta_1+b_j)\}z\\
	&= -(\delta_2+b')\{t_1(\delta_1+\delta_2+\widetilde{c_0}+c'-1)-(\delta_1+\widetilde{c_0})\}\{\prod_{j=1}^{n}(\delta_1+b_j)\}z.
\end{align*}
Hence we obtain
\begin{align*}
	&b't_1w_0 - t_2\sum_{j=1}^{n}w'_j\\
	&= [-b't_1(\delta_2+c'-1)+t_2(\delta_2+b')\{t_1(\delta_1+\delta_2+\widetilde{c_0}+c'-1)-(\delta_1+\widetilde{c_0})\}]\{\prod_{j=1}^{n}(\delta_1+b_j)\}z\\
	&= (t_1-t_2)(\delta_2+b')(\delta_1+\widetilde{c_0})\{\prod_{j=1}^{n}(\delta_1+b_j)\}z\\
	&= \frac{1}{b_1}(t_1-t_2)(\delta_1+b_1)w'_1,
\end{align*}
which is rewritten to equation \eqref{Eq:GA_deg_Pfaff_Proof_w'11}.

We can prove equations \eqref{Eq:GA_deg_Pfaff_Proof_wi1} and \eqref{Eq:GA_deg_Pfaff_Proof_w'i1} by substituting \eqref{Eq:GA_deg_Pfaff_Proof_w11} and \eqref{Eq:GA_deg_Pfaff_Proof_w'11} to
\[
	\widetilde{b_i}(\delta_1+\widetilde{c_{i-1}})w_{i-1} = \widetilde{b_{i-1}}(\delta_1+b_i)w_i,\quad
	\widetilde{b_i}(\delta_1+\widetilde{c_{i-1}})w'_{i-1} = \widetilde{b_{i-1}}(\delta_1+b_i)w'_i\quad (i=2,\ldots,n),
\]
which follows from definition \eqref{Eq:GA_deg_Pfaff}.

The rest three equations can be proved as follows.
We obtain equation \eqref{Eq:GA_deg_Pfaff_Proof_w01} by substituting \eqref{Eq:GA_deg_Pfaff_Proof_wi1} and \eqref{Eq:GA_deg_Pfaff_Proof_w'i1} to
\[
	b_1(\delta_1+\widetilde{c_0})w_0 = -(\delta_1+b_1)(w_1+w'_1).
\]
Then equation \eqref{Eq:GA_deg_Pfaff_Proof_w02} is given by substituting \eqref{Eq:GA_deg_Pfaff_Proof_w01} and \eqref{Eq:GA_deg_Pfaff_Proof_wi1} to
\[
	t_2\delta_2w_0 = -(t_2-1)\{(\delta_1+\widetilde{c_0}+b')w_0+\frac{1}{b_1}(\delta_1+b_1)w_1\}.
\]
We also obtain equation \eqref{Eq:GA_deg_Pfaff_Proof_w'i2} by substituting \eqref{Eq:GA_deg_Pfaff_Proof_w'i1} to
\[
	t_2\delta_2w'_i = -\delta_2w_i - (t_2-1)(\delta_1+\widetilde{c_0}-c'-1)w'_i\quad (i=1,\ldots,n).
\]
Note that three equations above follow from definition \eqref{Eq:GA_deg_Pfaff}.

\section{A degeneration of $\mathcal{H}_{n+1,2}$ and its particular solution in terms of $J_{2n}/I_{2n}$}\label{Sec:Sch_HGE_deg}

Let us substitute
\begin{equation}\label{Eq:Sch_Tsu_DS}
	q'_1 = 1 - q_1,\quad
	p'_1 = 0,\quad
	\kappa_0 = \kappa_1,
\end{equation}
to the system $\mathcal{H}_{n+1,2}$.
Then we obtain a Hamiltonian system of $(4n-2)$-nd order in two variables; denote it by $\mathcal{H}^{*}_{n+1,2}$.
It is derived from the isomonodromy deformation of the Fuchsian system of type $\{(n,1),(n,1),(n,1),(2,1,\ldots,1),(1,\ldots,1)\}$.
A degeneration of its hypergeometric solution occurs as follows.
\[\begin{array}{ccc}
	\mathcal{H}_{n+1,2} & \xrightarrow{\text{\small\eqref{Eq:Sch_Tsu_DS}}} & \mathcal{H}^{*}_{n+1,2} \\[8pt]
	\eqref{Eq:Sch_Tsu_PS_F2}\downarrow\qquad & & \qquad\downarrow\eqref{Eq:Sch_Tsu_deg_PS} \\
	\text{Pfaff system \eqref{Eq:Sch_Tsu_Pfaff}} & \xrightarrow{\text{\small\eqref{Eq:Sch_Tsu_Pfaff_DS}}} & \text{Pfaff system \eqref{Eq:Sch_Tsu_Pfaff_deg}} \\[8pt]
	\eqref{Eq:GA_deg_Pfaff}\uparrow\qquad & & \qquad\uparrow\eqref{Eq:GA_Pfaff} \\
	\text{HG system \eqref{Eq:Gen_Appell_2n+1}} & \xrightarrow{\text{\small Sec. \ref{Sec:Gen_Appell}}} & \text{HG system \eqref{Eq:Gen_Appell_2n}}
\end{array}\]
We state its detail in the following.

In system \eqref{Eq:Sch_Tsu_Pfaff} we assume that
\begin{equation}\label{Eq:Sch_Tsu_Pfaff_DS}
	w'_1 = (\kappa_1+\rho_1)w_0 - w_1,\quad
	\kappa_0 = \kappa_1.
\end{equation}
Then we obtain a linear Pfaff system for a vector of functions $\mathbf{w}={}^t[w_0,w_1,\ldots w_n,w'_2,\ldots,w'_n]$
\begin{equation}\label{Eq:Sch_Tsu_Pfaff_deg}
	d\mathbf{w} = \{A_1^{(1)}d\log(t_1-1)+A_0^{(1)}d\log t_1+A_td\log(t_1-t_2)+A_1^{(2)}d\log(t_2-1)+A_0^{(2)}d\log t_2\}\mathbf{w},
\end{equation}
with matrices
\begin{align*}
	A_1^{(1)} &= \theta_3E_{0,0}+\sum_{j=1}^{n}E_{0,j} + \sum_{i=1}^{n}\{\theta_3(\kappa_i+\rho_i)E_{i,0}+\sum_{j=1}^{n}(\kappa_i+\rho_i)E_{i,j}\},\\
	A_0^{(1)} &= -\sum_{j=2}^{n}E_{0,j}-\sum_{j=1}^{n-1}E_{0,n+j} + \sum_{i=1}^{n}\{(\rho_1-\rho_i)E_{i,i}-\sum_{j=i+1}^{n}(\kappa_i+\rho_i)E_{i,j}\}\\
	&\quad + \sum_{i=1}^{n-1}\{(\rho_1-\rho_{i+1})E_{n+i,n+i}-\sum_{j=i+1}^{n-1}(\kappa_{i+1}+\rho_{i+1})E_{n+i,n+j}\},\\
	A_t &= (\theta_2+\kappa_1+\rho_1)E_{0,0}-E_{0,1}+\sum_{j=1}^{n-1}E_{0,n+j}\\
	&\quad + \sum_{i=1}^{n-1}\{(\theta_2+\kappa_1+\rho_1)(\kappa_{i+1}+\rho_{i+1})E_{n+i,0}-(\kappa_{i+1}+\rho_{i+1})E_{n+i,1}+\sum_{j=1}^{n-1}(\kappa_{i+1}+\rho_{i+1})E_{n+i,n+j}\},\\
	A_1^{(2)} &= -\theta_3(\kappa_1+\rho_1)E_{1,0} + (\theta_2+\theta_3)E_{1,1} + \sum_{i=1}^{n-1}\{\theta_2E_{i+1,i+1}-\theta_3E_{i+1,n+i}\} + \sum_{i=1}^{n-1}\{-\theta_2E_{n+i,i+1}+\theta_3E_{n+i,n+i}\},\\
	A_0^{(2)} &= \sum_{i=1}^{n-1}\{-(\theta_2+\kappa_1+\rho_1)(\kappa_{i+1}+\rho_{i+1})E_{n+i,0}+(\kappa_{i+1}+\rho_{i+1})E_{n+i.1}+\theta_2E_{n+i,i+1}\}\\
	&\quad + \sum_{i=1}^{n-1}\{-\sum_{j=1}^{i-1}(\kappa_{i+1}+\rho_{i+1})E_{n+i,n+j}+(\theta_2+\kappa_1-\kappa_{i+1})E_{n+i,n+i}\},
\end{align*}
where $E_{i,j}$ stands for the $2n\times2n$ matrix with $1$ in the $(i,j)$-th entry and zeros elsewhere.
The Riemann scheme of system \eqref{Eq:Sch_Tsu_Pfaff_deg} is given by
\[
	\left\{\begin{array}{ccccccc}
		t_1=1 & t_1=0 & t_1=\infty & t_1=t_2 & t_2=1 & t_2=0 & t_2=\infty \\
		{[0]_{2n-1}} & {[0]_2} & {[-\kappa_2-\rho_1]_2} & {[0]_{2n-1}} & {[0]_n} & {[0]_{n+1}} & {[-\theta_2]_n} \\
		\theta_3' & {[\rho_1-\rho_2]_2} & \vdots & \theta_2' & {[\theta_2+\theta_3]_{n}} & \kappa_2' & \rho_1' \\
		& \vdots & {[-\kappa_n-\rho_1]_2} & & & \vdots & \vdots \\
		& {[\rho_1-\rho_n]_2} & -\kappa_1-\rho_1 & & & \kappa_n' & \rho_n' \\
		& & \rho_1' & & & & \\
	\end{array}\right\},
\]
where
\begin{align*}
	&\theta_2' = \theta_2 + \sum_{j=1}^{n}(\kappa_j+\rho_j),\quad
	\theta_3' = \theta_3 + \sum_{j=1}^{n}(\kappa_j+\rho_j),\quad
	\kappa_i' = \theta_2+\kappa_1-\kappa_i\quad (i=2,\ldots,n),\\
	&\rho_i' = -\theta_2-\theta_3-\kappa_1-\rho_i\quad (i=1,\ldots,n),
\end{align*}
Hence we can find that system \eqref{Eq:Sch_Tsu_Pfaff_deg} is the Pfaff one $J_{2n}/I_{2n}$.
Their spectral types are defined as follows.
\[\begin{array}{rl}
	I_{2n}: & \{(2n-1,1),(n,n),(n+1,1,\ldots,1),(n,1,\ldots,1)\} \\
	J_{2n}: & \{(2n-1,1),(2n-1,1),(2,\ldots,2),(2,\ldots,2,1,1)\}
\end{array}\]

In the system $\mathcal{H}^{*}_{n+1,2}$ we assume that
\begin{equation}\begin{split}\label{Eq:Sch_Tsu_deg_PS}
	&q_1p_1-\kappa_1-\rho_1 = 0,\quad
	q_ip_i-\kappa_i-\rho_i = 0\quad (i=2,\ldots,n),\\
	&p'_i = 0\quad (i=2,\ldots,n),\quad
	\theta_1 + \sum_{j=1}^{n}(\kappa_j+\rho_j) = 0.
\end{split}\end{equation}
We also define a dependent variable $w_0$ by
\begin{align*}
	d\log w_0 &= (\sum_{j=1}^{n}p_j+\theta_3)d\log(t_1-1) + \{-q_1p_1+\sum_{j=2}^{n}(q'_j-1)p_j+\kappa_1+\rho_1\}d\log t_1\\
	&\quad + \{(q_1-1)p_1-\sum_{j=2}^{n}q'_jp_j+\theta_2\}d\log(t_1-t_2),
\end{align*}
and set
\begin{align*}
	w_i &= p_iw_0\quad (i=1,\ldots,n), \\
	w'_i &= -q'_ip_iw_0\quad (i=2,\ldots,n).
\end{align*}
Then a vector of dependent variables $\mathbf{w}={}^t[w_0,w_1,\ldots w_n,w'_1,\ldots,w'_n]$ satisfies system \eqref{Eq:Sch_Tsu_Pfaff_deg}.

A solution of system \eqref{Eq:Sch_Tsu_Pfaff_deg} can be described in terms of the function $F_2^{(n-1)}$.
Let $\tilde{z}$ be a solution of system \eqref{Eq:Gen_Appell_2n}.
We also set
\begin{equation}\begin{split}\label{Eq:GA_Pfaff}
	w_0 &= -\{\prod_{j=2}^{n}(\delta_1+b_j)\}(\delta_2+c'-1)\tilde{z},\\
	w_1 &= -(c_1-1)(b'-c'+1)\{\prod_{j=2}^{n}(\delta_1+b_j)\}\tilde{z},\\
	w_i &= -(b_i-c_i+1)(b'-c'+1)\{\prod_{j=i+1}^{n}(\delta_1+b_j)\}\{\prod_{j=2}^{i-1}(\delta_1+c_j-1)\}\delta_1\tilde{z}\quad (i=2,\ldots,n),\\
	w'_i &= (b_i-c_i+1)\{\prod_{j=i+1}^{n}(\delta_1+b_j)\}\{\prod_{j=2}^{i-1}(\delta_1+c_j-1)\}\delta_1(\delta_2+b')\tilde{z}\quad (i=2,\ldots,n).
\end{split}\end{equation}
Then a vector of dependent variables $\mathbf{w}={}^t[w_0,w_1,\ldots w_n,w'_2,\ldots,w'_n]$ satisfies system \eqref{Eq:Sch_Tsu_Pfaff_deg} under a change of parameters
\begin{align*}
	&b_i = -\kappa_i-\rho_1\quad (i=2,\ldots,n),\quad
	b' = -\theta_2,\\
	&c_1 = -\kappa_1-\rho_1+1,\quad
	c_i = -\rho_1+\rho_i+1\quad (i=2,\ldots,n),\quad
	c' = 1-\theta_2-\theta_3.
\end{align*}

\section{A relationship between $F_2^{(n)}$ and $F_{n+1,2}$}\label{Sec:Rel_F2_F1}

The hypergeometric function $F_{n+1,2}$ given in \cite{T1} is an extension of Appell's $F_1$.
It is defined by a formal power series
\[
	F_{n+1,2}\left[\begin{array}{c}\alpha_1,\ldots,\alpha_n,\beta_1,\beta_2\\\gamma_1,\ldots,\gamma_n\end{array};s_1,s_2\right] = \sum_{i,j=0}^{\infty}\frac{(\alpha_1)_{i+j}\ldots(\alpha_n)_{i+j}(\beta_1)_i(\beta_2)_j}{(\gamma_1)_{i+j}\ldots(\gamma_n)_{i+j}(1)_i(1)_j}s_1^is_2^j.
\]
which is convergent in $|s_1|<1,|s_2|<1$.
This series satisfies a system of linear partial differential equations
\begin{equation}\begin{split}\label{Eq:Gen_Appell_F1}
	&[s_1(D_1+\beta_1)\{\prod_{i=1}^{n}(D+\alpha_i)\}-D_1\{\prod_{i=1}^{n}(D+\gamma_i-1)\}]y = 0,\\
	&[s_2(D_2+\beta_2)\{\prod_{i=1}^{n}(D+\alpha_i)\}-D_2\{\prod_{i=1}^{n}(D+\gamma_i-1)\}]y = 0,\\
	&\{s_1(D_1+\beta_1)D_2-s_2(D_2+\beta_2)D_1\}y = 0,\\
\end{split}\end{equation}
where $D_1=s_1\partial/\partial s_1$, $D_2=s_2\partial/\partial s_2$ and $D=D_1+D_2$.
System \eqref{Eq:Gen_Appell_F1} can be rewritten to a Pfaff one $I_{2n+1}/I_{2n+1}$; we do not state its detail.
An integral expression of $F_{n+1,2}$ is given by
\begin{equation}\label{Eq:Gen_Appell_F1_int}
	F_{n+1,2}\left[\begin{array}{c}\alpha_1,\ldots,\alpha_n,\beta_1,\beta_2\\\gamma_1,\ldots,\gamma_n\end{array};s_1,s_2\right] = \{\prod_{k=1}^{n}\frac{\Gamma(\gamma_k)}{\Gamma(\alpha_k)\Gamma(\gamma_k-\alpha_k)}\}\int_{0<v_n<\ldots<v_1<1}V(s)dv_1\ldots dv_n,
\end{equation}
where
\begin{align*}
	V(s) &= (1-v_1)^{\gamma_1-\alpha_1-1}\{\prod_{k=1}^{n-1}v_k^{\alpha_k-\gamma_{k+1}}(v_k-v_{k+1})^{\gamma_{k+1}-\alpha_{k+1}-1}\}(1-s_1v_n)^{-\beta_1}(1-s_2v_n)^{-\beta_2}v_n^{\alpha_n-1}.
\end{align*}
In this section we investigate a relationship between $F_2^{(n)}$ and $F_{n+1,2}$ from two points of view, namely a system of LPDEs and an integral expression.
Although we imposed the constraint $a=c_1+c'-2$ on $F_2^{(n)}$ in Section \ref{Sec:Gen_Appell}, we now assume that $a=c'$ in order to write some formulas concisely.

Let us consider a function
\[
	F_2^{(n)}\left[\begin{array}{c}b_1,\ldots,b_n,b',c'\\c_1,\ldots,c_n,c'\end{array};t_1,1-t_2\right] = \sum_{i,j=0}^{\infty}\frac{(b_1)_i\ldots(b_n)_i(b')_j(c')_{i+j}}{(c_1)_i\ldots(c_n)_i(c')_j(1)_i(1)_j}t_1^i(1-t_2)^j,
\]
which satisfies a system of LPDEs
\begin{equation}\begin{split}\label{Eq:Gen_Appell_F2}
	&[t_1(\delta_1+\delta_2+c')\{\prod_{i=1}^{n}(\delta_1+b_i)\}-\delta_1\{\prod_{i=1}^{n}(\delta_1+c_i-1)\}]z = 0,\\
	&\{(1-t_2)(\delta_1+\delta_2+c')(\delta_2+b')-\delta_2(\delta_2+c'-1)\}z = 0,\\
	&[t_1\delta_2\{\prod_{i=1}^{n}(\delta_1+b_i)\}+\{(1-t_2)(\delta_2+b')-\delta_2\}\{\prod_{i=1}^{n}(\delta_1+c_i-1)\}]z = 0.
\end{split}\end{equation}
Then we obtain the following theorem by a direct computation.

\begin{thm}\label{thm:Rel_F2_F1_Eq}
System \eqref{Eq:Gen_Appell_F2} is transformed by a transformation
\[
	s_1 = t_1,\quad
	s_2 = \frac{t_1}{t_2},\quad
	y = t_2^{b'}z,
\]
to system \eqref{Eq:Gen_Appell_F1} with
\[
	\alpha_i = b_i\quad (i=1,\ldots,n),\quad
	\beta_1 = c'-b',\quad
	\beta_2 = b',\quad
	\gamma_i = c_i\quad (i=1,\ldots,n).
\]
\end{thm}

\begin{rem}
In order to derive the formal power series of $F_{n+1,2}$ (resp. $F_2^{(n)}$) by the Frobenius method from system \eqref{Eq:Gen_Appell_F1} (resp. \eqref{Eq:Gen_Appell_F2}), the third equation of the system is unnecessary.
We add those equations in order to insist the equivalence between two system of LPDEs in Theorem \ref{thm:Rel_F2_F1_Eq}.
\end{rem}

\begin{cor}
Integral expression \eqref{Eq:Gen_Appell_int} with $a=c'$ is transformed to \eqref{Eq:Gen_Appell_F1_int} as
\[
	F_2^{(n)}\left[\begin{array}{c}\alpha_1,\ldots,\alpha_n,\beta_2,\beta_1+\beta_2\\\gamma_1,\ldots,\gamma_n,\beta_1+\beta_2\end{array};s_1,1-\frac{s_1}{s_2}\right] = s_1^{-b'}s_2^{b'}F_{n+1,2}\left[\begin{array}{c}\alpha_1,\ldots,\alpha_n,\beta_1,\beta_2\\\gamma_1,\ldots,\gamma_n\end{array};s_1,s_2\right],
\]
where $|s_1|<1$, $|s_2|<1$ and $|s_1|+|1-s_1/s_2|<1$.
\end{cor}

\begin{proof}
Assume that $a=c'$.
Then we have
\begin{align*}
	&\int_{1}^{t_2}(u_n-u_{n+1})^{-c'}(t_2-u_{n+1})^{c'-b'-1}(1-u_{n+1})^{b'-1}du_{n+1}\\
	&= \frac{\Gamma(b')\Gamma(c'-b')}{\Gamma(c')}(1-t_2)^{c'-1}(t_2-u_n)^{-b'}(1-u_n)^{b'-c'}.
\end{align*}
By using it we obtain
\[
	F_2^{(n)}\left[\begin{array}{c}b_1,\ldots,b_n,b',c'\\c_1,\ldots,c_n,c'\end{array};t_1,1-t_2\right] = t_1^{1-c_1}\{\prod_{k=1}^{n}\frac{\Gamma(c_k)}{\Gamma(b_k)\Gamma(c_k-b_k)}\}\int_{0<u_n<\ldots<u_1<t_1}U(t)du_1\ldots du_n,
\]
where
\begin{align*}
	U(t) &= (t_1-u_1)^{c_1-b_1-1}\{\prod_{k=1}^{n-1}u_k^{b_k-c_{k+1}}(u_k-u_{k+1})^{c_{k+1}-b_{k+1}-1}\}(t_2-u_n)^{-b'}(1-u_n)^{b'-c'}u_n^{b_n-1}.
\end{align*}
It is transformed by
\[
	u_k = s_1v_k\quad (k=1,\ldots,n),\quad
	t_1 = s_1,\quad
	t_2 = s_1/s_2,
\]
to integral expression \eqref{Eq:Gen_Appell_F1_int} with
\[
	\alpha_i = b_i\quad (i=1,\ldots,n),\quad
	\beta_1 = -b'+c',\quad
	\beta_2 = b',\quad
	\gamma_i = c_i\quad (i=1,\ldots,n).
\]
\end{proof}

\section{Conclusion}\label{Sec:Conclusion}

In this article we introduced the hypergeometric function $F_2^{(n)}$ as an extension of Appell's $F_2$.
It was obtained from a particular solution of the Painlev\'e system $\mathcal{H}^{*}_{n+1,2}$.
We also shew that $F_2^{(n)}$ with one special parameter is equivalent to the hypergeometric function $F_{n+1,2}$ which is an extension of Appell's $F_1$.
Summarizing the obtained result, we can draw the following picture.
\[\begin{array}{ccccccccc}
	\downarrow & & \downarrow & & \downarrow & & \downarrow & & \downarrow \\[4pt]
	F_2^{(n)} & \longrightarrow & I_{2n+2}/J_{2n+2} & \longleftarrow & \mathcal{H}^*_{n+2,2} & \longrightarrow & I_{2n+2}/I_{2n+2} & \longleftarrow & F_{n+2,2}| \\[4pt]
	\downarrow & & \downarrow & & \downarrow & & \downarrow & & \downarrow \\[4pt]
	F_2^{(n)}| & \longrightarrow & I_{2n+1}/J_{2n+1} & \longleftarrow & \mathcal{H}_{n+1,2} & \longrightarrow & I_{2n+1}/I_{2n+1} & \longleftarrow & F_{n+1,2} \\[4pt]
	\downarrow & & \downarrow & & \downarrow & & \downarrow & & \downarrow \\[4pt]
	F_2^{(n-1)} & \longrightarrow & I_{2n}/J_{2n} & \longleftarrow & \mathcal{H}^*_{n+1,2} & \longrightarrow & I_{2n}/I_{2n} & \longleftarrow & F_{n+1,2}| \\[4pt]
	\downarrow & & \downarrow & & \downarrow & & \downarrow & & \downarrow \\
	\vdots & & \vdots & & \vdots & & \vdots & & \vdots \\
	\downarrow & & \downarrow & & \downarrow & & \downarrow & & \downarrow \\[4pt]
	F_2 & \longrightarrow & I_4/I_4 & \longleftarrow & \mathcal{H}^*_{3,2} & \longrightarrow & I_4/I_4 & \longleftarrow & F_{3,2}| \\[4pt]
	\downarrow & & \downarrow & & \downarrow & & \downarrow & & \downarrow \\[4pt]
	F_2| & \longrightarrow & P_3/P_3 & \longleftarrow & \mathcal{H}_{2,2} & \longrightarrow & P_3/P_3 & \longleftarrow & F_1 \\[4pt]
\end{array}\]
The symbol $F_2^{(n)}|$ (resp. $F_{n+1,2}|$) stands for the function $F_2^{(n)}$ (resp. $F_{n+1,2}$) with one special parameter.
In the above picture the bottom is given by the classical result \cite{B,Fuc,G,Sla}, the part $\mathcal{H}_{n+1,2}\longrightarrow I_{2n+1}/I_{2n+1}\longleftarrow F_{n+1,2}$ by \cite{T1,T3} and the rest by this article.

The result of this article gives an example that one Painlev\'e system contains two types of rigid systems.
It suggests that if two Painlev\'e systems $\mathcal{H}^{n1,\lambda_2,\ldots,\lambda_{m+3}}$ and $\mathcal{H}^{n1,\mu_2,\ldots,\mu_{m+3}}$ are equivalent, then they contains two rigid systems, of type $\lambda_2,\ldots,\lambda_{m+3}$ and of type $\mu_2,\ldots,\mu_{m+3}$.
Here we let $\lambda_i,\mu_i$ $(i=2,\ldots,m+3)$ be partitions of a natural number $n+1$.
It is also suggested that properties of Painlev\'e systems, especially affine Weyl group symmetries, can be applied to analyses of rigid systems and hypergeometric functions.
As is seen in Remark \ref{Rem:Rel_F2_F1}, the relationship between $F_2^{(n)}$ and $F_{n+1,2}$ is derived from the symmetry of $\mathcal{H}_{n+1,2}$.

On the other hand, in \cite{T1}, the function $F_{n+1,2}$ was extended to the one in multi-variables.
Namely, we have a formal power series
\[
	F_{n+1,m} = \sum_{i_1,\ldots,i_m\in\mathbb{Z}_{\geq0}}\frac{(a_1)_{|i|}\ldots(a_n)_{|i|}(b_1)_{i_1}\ldots(b_m)_{i_m}}{(c_1)_{|i|}\ldots(c_n)_{|i|}(1)_{i_1}\ldots(1)_{i_m}}t_1^{i_1}\ldots t_m^{i_m},
\]
where $|i|=i_1+\ldots+i_m$, as an extension of $F_{n+1,2}$ and Lauricella's $F_D$.
Therefore it is a natural motivation to propose an extension of $F_2^{(n)}$ and Lauricella's $F_A$.
Although we can nominate a formal power series
\[
	F_2^{(n,m)} = \sum_{i_1,\ldots,i_m\in\mathbb{Z}_{\geq0}}\frac{(b_{1,1})_{i_1}\ldots(b_{1,n})_{i_1}(b_2)_{i_2}\ldots(b_m)_{i_m}(a)_{|i|}}{(c_{1,1})_{i_1}\ldots(c_{1,n})_{i_1}(c_2)_{i_2}\ldots(c_m)_{i_m}(1)_{i_1}\ldots(1)_{i_m}}t_1^{i_1}\ldots t_m^{i_m},
\]
we have not found any suitable rigid system or Painlev\'e system even in the case of $m=3$.
It is a future problem.
Also, we have to consider an extension of Appell's $F_4$ and Lauricella's $F_C$.
In fact, the function $F_4$ has already appeared in a particular solution of a Painlev\'e system; see Appendix \ref{App:Appell_F4}.

\begin{rem}
The function $F_A$ in $m$ variables, namely $F_2^{(1,m)}$, satisfies a system of partial differential equations
\[
	\{t_i(D+a)(\delta_i+b_i)-\delta_i(\delta_i+c_i-1)\}z = 0\quad (i=1,\ldots,m),
\]
where $\delta_i=t_i\partial/\partial t_i$ and $D=\delta_1+\ldots+\delta_m$.
It can be rewritten to a linear Pfaff system with $2^m\times2^m$ matrices
\[
	d\mathbf{w} = \{\sum_{i=1}^{m}A_id\log t_i+\sum_{\{i_1,\ldots,i_k\}\subset\{1,\ldots,m\}}A_{i_1\ldots i_k}d\log(t_{i_1}+\ldots+t_{i_k}-1)\}\mathbf{w}.
\]
If, in the case of $m=3$, we regard each of $t_1,t_2,t_3$ as an independent variable and the others as constants, then this Pfaff system becomes a Fuchsian one of type $71,71,71,71,44,41111$.
Interestingly this type of Fuchsian system is reduced by Katz's two operations to a Fuchsian system of type $21,21,21,21,111$, which becomes an origin of the 6th order Painlev\'e system $\mathcal{H}^{21,21,21,21,111}$ whose particular solution is described in terms of $F_2=F_2^{(1,2)}$.
\end{rem}

\appendix

\section{A proof of integral expressions \eqref{Eq:Gen_Appell_int} and \eqref{Eq:Gen_Appell_deg_int}}\label{App:Gen_Appell_int}

We first prove integral expression \eqref{Eq:Gen_Appell_int}.
Let us recall that
\[
	\frac{(b_k)_i}{(c_k)_i} = \frac{\Gamma(b_k+i)\Gamma(c_k)}{\Gamma(b_k)\Gamma(c_k+i)} = \frac{\mathrm{B}(b_k+i,c_k-b_k)}{\mathrm{B}(b_k,c_k-b_k)},
\]
and
\[
	\mathrm{B}(b,c) = \int_{0}^{1}u^{b-1}(1-u)^{c-1}du.
\]
Then the formal power series of $F_2^{(n)}$ is transformed to a multiple integral as
\begin{align*}
	&F_2^{(n)}\left[\begin{array}{c}b_1,\ldots,b_n,b',a\\c_1,\ldots,c_n,c'\end{array};t_1,1-t_2\right]\\
	&= \sum_{i,j=0}^{\infty}\{\prod_{k=1}^{n}\frac{\mathrm{B}(b_k+i,c_k-b_k)}{\mathrm{B}(b_k,c_k-b_k)}\}\frac{\mathrm{B}(b'+j,c'-b')}{\mathrm{B}(b',c'-b')}\frac{(a)_{i+j}}{(1)_i(1)_j}t_1^i(1-t_2)^j\\
	&= \{\prod_{k=1}^{n}\frac{1}{\mathrm{B}(b_k,c_k-b_k)}\}\frac{1}{\mathrm{B}(b',c'-b')}\sum_{i,j=0}^{\infty}\int_{[0,1]^{n+1}}\{\prod_{k=1}^{n}v_k^{b_k+i-1}(1-v_k)^{c_k-b_k-1}\}v_{n+1}^{b'+j-1}(1-v_{n+1})^{c'-b'-1}\\
	&\quad \times \frac{(a)_{i+j}}{(1)_i(1)_j}t_1^i(1-t_2)^jdv_1\ldots dv_{n+1}\\
	&= \{\prod_{k=1}^{n}\frac{1}{\mathrm{B}(b_k,c_k-b_k)}\}\frac{1}{\mathrm{B}(b',c'-b')}\int_{[0,1]^{n+1}}\{\prod_{k=1}^{n}v_k^{b_k-1}(1-v_k)^{c_k-b_k-1}\}v_{n+1}^{b'-1}(1-v_{n+1})^{c'-b'-1}\\
	&\quad \times \{1-t_1v_1\ldots v_n-(1-t_2)v_{n+1}\}^{-a}dv_1\ldots dv_{n+1}.
\end{align*}
Hence we obtain integral expression \eqref{Eq:Gen_Appell_int} by a transformation
\[
	u_k = t_1v_1\ldots v_k\quad (k=1,\ldots,n),\quad
	u_{n+1} = 1 - (1-t_2)v_{n+1}.
\]

We next prove integral expression \eqref{Eq:Gen_Appell_deg_int}.
Similarly as above we have
\begin{equation}\begin{split}\label{Eq:Gen_Appell_deg_int_Proof_1}
	&F_2^{(n)}\left[\begin{array}{c}b_1,\ldots,b_n,b'-c'+1,c_1-1\\c_1-1,c_2,\ldots,c_n,2-c'\end{array};t_1,1-t_2\right]\\
	&= \sum_{i,j=0}^{\infty}\{\prod_{k=2}^{n}\frac{\mathrm{B}(b_k+i,c_k-b_k)}{\mathrm{B}(b_k,c_k-b_k)}\}\frac{\mathrm{B}(b'-c'+1+j,1-b')}{\mathrm{B}(b'-c'+1,1-b')}\frac{(c_1-1+i)_j(b_1)_{i+j}}{(b_1+i)_j(1)_i(1)_j}t_1^i(1-t_2)^j\\
	&= \{\prod_{k=2}^{n}\frac{1}{\mathrm{B}(b_k,c_k-b_k)}\}\frac{1}{\mathrm{B}(b'-c'+1,1-b')}\int_{[0,1]^{n+1}}\{\prod_{k=1}^{n-1}v_k^{b_{k+1}-1}(1-v_k)^{c_{k+1}-b_{k+1}-1}\}v_n^{b'-c'}(1-v_n)^{-b'}\\
	&\quad \times \sum_{i,j=0}^{\infty}\frac{(c_1-1+i)_j(b_1)_{i+j}}{(b_1+i)_j(1)_i(1)_j}(t_1v_1\ldots v_{n-1})^i\{(1-t_2)v_n\}^jdv_1\ldots dv_n.
\end{split}\end{equation}
On the other hand, by using an equation
\[
	\frac{(c_1-1+i)_j}{(b_1+i)_j} = \sum_{k=0}^{j}\frac{(c_1-b_1-1)_k}{(b_1+i+j-k)_k}\frac{(1)_j}{(1)_k(1)_{j-k}},
\]
we have
\begin{equation}\begin{split}\label{Eq:Gen_Appell_deg_int_Proof_2}
	&\sum_{i,j=0}^{\infty}\frac{(c_1-1+i)_j(b_1)_{i+j}}{(b_1+i)_j(1)_i(1)_j}(t_1v_1\ldots v_{n-1})^i\{(1-t_2)v_n\}^j\\
	&= \sum_{i,j=0}^{\infty}\sum_{k=0}^{j}\frac{(b_1)_{i+j}(c_1-b_1-1)_k}{(b_1+i+j-k)_k(1)_i(1)_k(1)_{j-k}}(t_1v_1\ldots v_{n-1})^i\{(1-t_2)v_n\}^j\\
	&= \sum_{i=0}^{\infty}\frac{1}{(1)_i}(t_1v_1\ldots v_{n-1})^i\sum_{j=0}^{\infty}\sum_{k=0}^{j}\frac{(b_1)_{i+j-k}(c_1-b_1-1)_k}{(1)_k(1)_{j-k}}\{(1-t_2)v_n\}^j\\
	&= \sum_{i=0}^{\infty}\frac{1}{(1)_i}(t_1v_1\ldots v_{n-1})^i\sum_{k=0}^{\infty}\sum_{l=0}^{\infty}\frac{(b_1)_{i+l}(c_1-b_1-1)_k}{(1)_k(1)_l}\{(1-t_2)v_n\}^{k+l}\\
	&= \sum_{i,j=0}^{\infty}\frac{(b_1)_{i+j}}{(1)_i(1)_j}(t_1v_1\ldots v_{n-1})^i\{(1-t_2)v_n\}^j\sum_{k=0}^{\infty}\frac{(c_1-b_1-1)_k}{(1)_k}\{(1-t_2)v_n\}^k\\
	&= \{1-t_1v_1\ldots v_{n-1}-(1-t_2)v_n\}^{-b_1}\{1-(1-t_2)v_n\}^{b_1-c_1+1}.
\end{split}\end{equation}
Substituting equation \eqref{Eq:Gen_Appell_deg_int_Proof_2} to \eqref{Eq:Gen_Appell_deg_int_Proof_1}, we obtain
\begin{align*}
	&F_2^{(n)}\left[\begin{array}{c}b_1,\ldots,b_n,b'-c'+1,c_1-1\\c_1-1,c_2,\ldots,c_n,c'\end{array};t_1,1-t_2\right]\\
	&= \{\prod_{k=2}^{n}\frac{1}{\mathrm{B}(b_k,c_k-b_k)}\}\frac{1}{\mathrm{B}(b'-c'+1,1-b')}\int_{[0,1]^{n+1}}\{\prod_{k=1}^{n-1}v_k^{b_{k+1}-1}(1-v_k)^{c_{k+1}-b_{k+1}-1}\}v_n^{b'-c'}(1-v_n)^{-b'}\\
	&\quad \times \{1-t_1v_1\ldots v_{n-1}-(1-t_2)v_n\}^{-b_1}\{1-(1-t_2)v_n\}^{b_1-c_1+1}dv_1\ldots dv_n.
\end{align*}
It is transformed to integral expression \eqref{Eq:Gen_Appell_deg_int} by
\[
	u_k = t_1v_1\ldots v_k\quad (k=1,\ldots,n-1),\quad
	u_n = 1 - (1-t_2)v_n.
\]

\section{The higher order Painlev\'e system $\mathcal{H}^{31,31,22,22,22}$ and the Appell hypergeometric function $F_4$}\label{App:Appell_F4}

The Painlev\'e system $\mathcal{H}^{31,31,22,22,22}$ given in \cite{Su2} is expressed as the Hamiltonian system
\[
	t_i(t_i-1)\frac{\partial q_j}{\partial t_i} = \frac{\partial H_i}{\partial p_j},\quad
	t_i(t_i-1)\frac{\partial p_j}{\partial t_i} = -\frac{\partial H_i}{\partial q_j},\quad (i\in\mathbf{Z}/2\mathbf{Z};j=1,2,3),
\]
with
\begin{align*}
	H_i &= H^{Gar}_i - (\alpha_1-\alpha_2)\frac{t_1}{t_1-t_2}q_2\{(t_i-1)p_i-(t_2-1)p_{i+1}\}\\
	&\quad + \delta_{i,2}(\alpha_1-\alpha_2)t_2(q_2-1)p_2 - (t_i+1)q_3^2p_3^2 + \{\alpha_1+\alpha_3-1+(\alpha_1+\alpha_4)t_i\}q_3p_3\\
	&\quad + q_iq_3p_3(2q_1p_1+2q_2p_2+q_3p_3+2\alpha_0+\alpha_5+1) + t_iq_{i+1}p_3(q_3p_3-\alpha_1+\alpha_2)\\
	&\quad - 2(t_i+1)q_iq_3p_ip_3 - q_3p_{i+1}(2q_ip_i+q_{i+1}p_{i+1}+2q_3p_3+2\alpha_0+\alpha_2+\alpha_5+1)\\
	&\quad + 2t_iq_3p_ip_3 + \frac{1}{t_i-t_{i+1}}q_3\{t_i(t_i-1)p_i^2-2t_i(t_{i+1}-1)p_1p_2+(t_i-1)t_{i+1}p_{i+1}^2\},
\end{align*}
where $H^{Gar}_i$ is the Hamiltonian of the Garnier system in two variables
\begin{align*}
	H^{Gar}_i &= q_i(q_i-1)(q_i-t_i)p_i^2 - (\alpha_{i+1}+\alpha_3-1)q_i(q_i-1)p_i - \alpha_4q_i(q_i-t_i)\\
	&\quad - \alpha_i(q_i-1)(q_i-t_i)p_i + \alpha_0(\alpha_0+\alpha_5+1)q_i\\
	&\quad + q_1q_2p_{i+1}(2q_ip_i+q_{i+1}p_{i+1}+2\alpha_0+\alpha_5+1)\\
	&\quad - \frac{1}{t_i-t_{i+1}}q_1q_2\{t_i(t_i-1)p_i^2-2t_i(t_{i+1}-1)p_1p_2+(t_i-1)t_{i+1}p_{i+1}^2\}\\
	&\quad + \alpha_i\frac{t_i}{t_i-t_{i+1}}q_{i+1}\{(t_i-1)p_i-(t_{i+1}-1)p_{i+1}\} - \alpha_{i+1}\frac{(t_i-1)t_{i+1}}{t_i-t_{i+1}}q_i(p_i-p_{i+1}),
\end{align*}
and parameters satisfy a relation $2\alpha_0+\alpha_1+\alpha_2+\alpha_3+\alpha_4+\alpha_5=0$.

In the system $\mathcal{H}^{31,31,22,22,22}$ we assume that
\[
	q_1 = q_2 = q_3 = \alpha_1 = 0.
\]
We also define a dependent variable $w_0$ by
\[
	\frac{\partial}{\partial t_i}\log w_0 = \frac{p_i-(\alpha_0+\alpha_5+1)}{t_i-1} - \frac{\alpha_3}{t_i}\quad (i\in\mathbf{Z}/2\mathbf{Z}),
\]
and set
\[
	\frac{w_1}{w_0} = -t_1p_1,\quad
	\frac{w_2}{w_0} = -t_2p_2,\quad
	\frac{w_3}{w_0} = t_1t_2(p_1p_2-\alpha_2p_3).
\]
Then we have

\begin{thm}[\cite{Su2}]
A vector of dependent variables $\mathbf{w}={}^t[w_0,w_1,w_2,w_3]$ defined above satisfies a linear Pfaff system
\begin{equation}\label{Eq:F4_Pfaff}
	d\mathbf{w} = \{A_1^{(1)}d\log(t_1-1)+A_0^{(1)}d\log t_1+A_td\log(t_1-t_2)+A_1^{(2)}d\log(t_2-1)+A_0^{(2)}d\log t_2\}\mathbf{w},
\end{equation}
with matrices
\begin{align*}
	&A_1^{(1)} = \begin{bmatrix}-(\alpha_0+\alpha_5+1)&-1&0&0\\\alpha_0(\alpha_0+\alpha_5+1)&\alpha_0&0&0\\0&0&-(\alpha_0+\alpha_2+\alpha_5+1)&-1\\0&0&(\alpha_0+\alpha_2)(\alpha_0+\alpha_2+\alpha_5+1)&\alpha_0+\alpha_2\end{bmatrix},\\
	&A_0^{(1)} = \begin{bmatrix}-\alpha_3&1&0&0\\0&0&0&0\\0&\alpha_2&-\alpha_3&1\\0&0&0&0\end{bmatrix},\quad
	A_t = \begin{bmatrix}0&0&0&0\\0&\alpha_2&-\alpha_2&0\\0&-\alpha_2&\alpha_2&0\\0&0&0&0\end{bmatrix},
\end{align*}
and
\[
	A_1^{(2)} = E_{23}A_1^{(1)}E_{23},\quad
	A_0^{(2)} = E_{23}A_0^{(1)}E_{23},\quad
	E_{23} = \begin{bmatrix}1&0&0&0\\0&0&1&0\\0&1&0&0\\0&0&0&1\end{bmatrix}.
\]
\end{thm}

The Riemann scheme of system \eqref{Eq:F4_Pfaff} is given by
\[
	\left\{\begin{array}{cccc}
		t_1=1,t_2=1 & t_1=0,t_2=0 & t_1=\infty,t_2=\infty & t_1=t_2 \\
		-\alpha_5-1 & -\alpha_3 & \alpha_0+\alpha_3+\alpha_5+1 & 2\alpha_2 \\
		-\alpha_5-1 & -\alpha_3 & \alpha_0+\alpha_3+\alpha_5+1 & 0 \\
		0 & 0 & -\alpha_0-\alpha_2 & 0 \\
		0 & 0 & -\alpha_0-\alpha_2 & 0 \\
	\end{array}\right\}.
\]
Hence we can find that system \eqref{Eq:F4_Pfaff} is the Pfaff one $P_{4,4}/P_{4,4}$.
Namely, if we regard each of $t_1,t_2$ as an independent variable and another as a constant, then this Pfaff system becomes a rigid one of type $31,22,22,22$.

The solution of \eqref{Eq:F4_Pfaff} can be described in terms of the Appell hypergeometric function $F_4$.
In \cite{Kato} it was shown that $F_4$ satisfies a system of linear partial differential equations
\begin{equation}\begin{split}\label{Eq:F4}
	&\{t_1(\delta_1+a)(\delta_1+b)-\delta_1(\delta_1+c_1)+\frac{(a+b-c_1-c_2+1)t_1(t_2-1)(\delta_1-\delta_2)}{t_1-t_2}\}z = 0,\\
	&\{t_2(\delta_2+a)(\delta_2+b)-\delta_2(\delta_2+c_1)+\frac{(a+b-c_1-c_2+1)t_2(t_1-1)(\delta_2-\delta_1)}{t_2-t_1}\}z = 0.
\end{split}\end{equation}
Thanks to this fact, we obtain

\begin{thm}
Let $z$ be a dependent variable satisfying system \eqref{Eq:F4} with parameters
\[
	a = \alpha_0+\alpha_3+\alpha_5+1,\quad
	b = -\alpha_0-\alpha_2,\quad
	c_1 = \alpha_3+1,\quad
	c_2 = \alpha_5+1,
\]
We also set
\begin{align*}
	w_1 &= -[(\delta_1-\alpha_0)(\delta_2-\alpha_0)+\frac{\alpha_2\{t_2(\delta_1-\alpha_0)-t_1(\delta_2-\alpha_0)\}}{t_1-t_2}]z,\\
	w_2 &= \alpha_0[(\delta_1+\alpha_3)(\delta_2-\alpha_0)+\frac{\alpha_2\{t_2(\delta_1+\alpha_3)-t_1(\delta_2+\alpha_3)\}}{t_1-t_2}]z,\\
	w_3 &= \alpha_0[(\delta_1-\alpha_0)(\delta_2+\alpha_3)+\frac{\alpha_2\{t_2(\delta_1+\alpha_3)-t_1(\delta_2+\alpha_3)\}}{t_1-t_2}]z,\\
	w_4 &= -\alpha_0(\alpha_0+\alpha_2)[(\delta_1+\alpha_3)(\delta_2+\alpha_3)+\frac{\alpha_2\{t_2(\delta_1+\alpha_3)-t_1(\delta_2+\alpha_3)\}}{t_1-t_2}]z.
\end{align*}
Then a vector of dependent variables $\mathbf{w}={}^t[w_0,w_1,w_2,w_3]$ satisfies system \eqref{Eq:F4_Pfaff}.
\end{thm}

\begin{rem}
In \cite{Kato} the linear Pfaff system with $4\times4$ matrices was also given.
It is transformed to system \eqref{Eq:F4_Pfaff} via a certain gauge transformation.
\end{rem}

\section*{Acknowledgement}

The author would like to express his gratitude to Professors Akihito Ebisu, Yoshishige Haraoka, Tom Koornwinder, Masatoshi Noumi, Hiroyuki Ochiai, Yohsuke Ohyama, Toshio Oshima, Hidetaka Sakai, Teruhisa Tsuda and Yasuhiko Yamada for helpful comments and advices.
This work was supported by JSPS KAKENHI Grant Number 15K04911.


\end{document}